\documentclass[12pt]{amsart}
\usepackage{etex}
\usepackage{latexsym}
\usepackage{amsthm,amssymb,enumitem,mathtools}

\usepackage{tikz}

\usepackage{scalerel}

\usetikzlibrary{arrows,shapes,automata,backgrounds,decorations,petri,positioning}
\usetikzlibrary[calc,intersections,through,backgrounds,arrows,decorations.pathmorphing]

\tikzstyle{edge} = [fill,opacity=.5,fill opacity=.5,line cap=round, line join=round, line width=50pt]

\pgfdeclarelayer{background}
\pgfsetlayers{background,main}

\usepackage[margin=1in,letterpaper,portrait]{geometry}

\usepackage{amsthm}

\usepackage[latin1]{inputenc}
\usepackage{subcaption}
\usepackage{color}
\usepackage{amsmath}
\usepackage{amsthm}
\usepackage{amstext}
\usepackage{amssymb}
\usepackage{amsfonts}
\usepackage{graphicx}
\usepackage{young}
\usepackage{multicol}
\usepackage{mathrsfs}
\usepackage[all]{xy}
\usepackage{tikz}
\usepackage{mathtools}
\usepackage{tabularx}
\usepackage{array}
\usepackage{commath}
\usepackage{ytableau}

\usepackage{scrextend}

\usepackage{hyperref}

\theoremstyle{definition}
\newtheorem{theorem}{Theorem}[section]
\newtheorem{conjecture}[theorem]{Conjecture}
\newtheorem{remark}[theorem]{Remark}
\newtheorem{lemma}[theorem]{Lemma}
\newtheorem{definition}[theorem]{Definition}

\newtheorem{example}[theorem]{Example}
\newtheorem{proposition}[theorem]{Proposition}
\newtheorem{corollary}[theorem]{Corollary}

\newtheorem{thm}{Theorem}

\theoremstyle{remark}

\newtheorem*{ackn}{Acknowledgments}

\DeclareMathAlphabet{\mathpzc}{OT1}{pzc}{m}{it}

\setlist[description]{font=\normalfont\itshape\space}

\newcommand{\lcov}{\vartriangleleft}

\newcommand{\symm}{\mathfrak{S}}

\newcommand{\odd}{D_o}
\newcommand{\swap}{\overset{\circ}{\leftrightarrow}}
\newcommand{\lra}{\leftrightarrow}
\DeclareMathOperator{\avoid}{Av}
\DeclareMathOperator{\Av}{Av}
\DeclareMathOperator{\perm}{Perm}
\DeclareMathOperator{\Perm}{Perm}
\DeclareMathOperator{\flip}{flip}
\newcommand{\ub}[1]{\!\underbracket[.5pt][1pt] {#1}}
\newcommand{\ol}[1]{\overline {#1}}

\begin{document}

\title{Odd diagrams, Bruhat order, and pattern avoidance}%

\date{}

\author[F.~Brenti]{Francesco Brenti}\address{Dipartimento di Matematica 
					Universit\`{a} di Roma ``Tor Vergata''
					Via della Ricerca Scientifica, 1 
					00133 Roma, Italy }
	\email{brenti@mat.uniroma2.it}
	
\author[A.~Carnevale]{Angela Carnevale}
\address{School of Mathematical and Statistical Sciences, National University of Ireland,
  Galway, Ireland}
	\email{angela.carnevale@nuigalway.ie}
	
\author[B.~E.~Tenner]{Bridget Eileen Tenner}
\address{Department of Mathematical Sciences, DePaul University, Chicago, IL, USA}
\email{bridget@math.depaul.edu}



\begin{abstract}%
  The odd diagram of a permutation is a subset of the classical
  diagram with additional parity conditions. In this paper, we study
  classes of permutations with the same odd diagram, which we call odd
  diagram classes. First, we prove a conjecture relating odd diagram
  classes and 213- and 312-avoiding permutations. Secondly, we show
  that each odd diagram class is a Bruhat interval. Instrumental to
  our proofs is an explicit description of the Bruhat edges that link
  permutations in a class.
\end{abstract}

\maketitle
\section{Introduction}

Odd analogues of well-known combinatorial objects and statistics
associated with permutations (and, more generally, with Weyl and
Coxeter group elements) have been recently considered and studied
(see, for instance,
\cite{BC2/17,BC/17,BC/19,BC/20,BS/20,KV,SV/13,Stembridge/19,Sun/18}).
In particular, odd analogues of permutation diagrams, called \emph{odd
  diagrams}, were introduced and studied in~\cite{BC/20}. It is well
known that (classical) diagrams of permutations are in bijection with
permutations themselves, and they capture interesting related
features. For instance, the size of the diagram of a permutation
equals its number of inversions (that is, its length), and the diagram
constitutes a main tool in the definition of the Schubert variety
associated with the permutation.  The odd diagram of a permutation
(see Definition~\ref{def odd diag}) is a subset of the diagram, its
size equals the number of \emph{odd inversions}---its \emph{odd
  length}, and can be used to define a corresponding \emph{odd
  Schubert variety} \cite{BC/20}.

It is easy to see that odd diagrams are not in bijection with
permutations. As we show in this paper, however, when passing from
diagrams to odd diagrams, we trade faithful encodings of permutations
for a rich combinatorial structure within each \emph{odd diagram
  class} (for a definition see Section~\ref{sec:injectivity}).  In
this article, we carry out an in-depth analysis of these classes,
relating them to well-studied notions such as pattern avoidance and
Bruhat order in symmetric groups.

The following is our first main result, which implies in particular
\cite[Conjecture 6.1]{BC/20}.
\begin{thm}\label{thm:avoid conjecture}
Every odd diagram class contains at
 most  one permutation avoiding the pattern $213$ and at most one avoiding
 $312$.  If these permutations  exist, they are, respectively, the maximum and the
 minimum  elements of the class with respect to the Bruhat order.
\end{thm}

This is proved in Sections~\ref{sec:injectivity} and \ref{sec:legal moves}
(cf.\ Corollaries~\ref{cor:injective} and~\ref{cor:length conjecture})
as a consequence of a detailed analysis of \emph{legal moves} carried
out in Section~\ref{sec:legal moves}, and a certain notion of
\emph{connectivity} within odd diagram classes.

Our second main result shows that odd diagrams partition each
symmetric group in a particularly pleasant way (see also
Theorem~\ref{thm:bruhat}).

\begin{thm}\label{thm:bruhat intro}
The subset of $\symm_n$ having a given odd diagram is a Bruhat interval.
\end{thm}

The characterization of legal moves and Theorem~\ref{thm:bruhat intro}
imply that, with respect to the right weak order, odd diagram classes
exhibit the opposite behavior. Namely, in right weak order each odd
diagram class is an antichain (cf.\ Corollary \ref{cor:antichain}).

\vspace{1em}

The paper is organized as follows. In Section~\ref{sec:notation}, we
collected some notation, definitions and preliminaries. The next
sections focus on relationships between permutations having the same
odd diagram. As we will see in Section~\ref{sec:injectivity}, a key
role in the study of this property is played by the permutation
patterns $213$ and $312$. This sets the stage for
Section~\ref{sec:legal moves}, which develops a method for traversing
the permutations in an odd diagram class. In this section we also
complete the proof of Theorem~\ref{thm:avoid conjecture}, resolving in
particular the length conjecture \cite[Conjecture 6.1]{BC/20}. In
Section~\ref{sec:length}, we analyze the consequences of legality and study 
so-called \emph{illegal patterns}.  We prove Theorem~\ref{thm:bruhat
  intro} in the last section. We show that each odd diagram class has
a unique Bruhat minimal and a unique Bruhat maximal element and we
build on the theory of legal moves to show that in an odd diagram
class there is always a maximal chain of elements within the class. We
conclude the paper with some remarks and open questions regarding the
intervals arising from odd diagram classes.

\medskip


\section{Notation and preliminaries}\label{sec:notation}


For $n\in\mathbb N$, we let $\symm_n$ denote the symmetric group of
degree $n$. We regard $\symm_n$ as a Coxeter group generated by the
simple transpositions $S=\{(i\, i+1):i=1,\ldots n-1\}$. The set of
reflections of $\symm_n$ is
$T=\{w^{-1}sw:w\in \symm_n,\, s\in S\}=\{(a\,b):1\leq a<b\leq
n\}$. The Coxeter length of a permutation $w\in\symm_n$ is denoted
$\ell(w)$.  It is well known (see, e.g., \cite[Proposition 1.5.2]{BB})
that
\[
\ell(w)=| \{ (i,j) \in [n]^2 : i<j,\,  w(i)>w(j) \}|.
\]

The \emph{Bruhat graph} of $\symm_n$ is the directed graph
$B(\symm_n)$ having $\symm_n$ as its vertex set and where, for
$u,v \in \symm_n$, $u \rightarrow v$ if and only if $v\, u^{-1} \in T$
and $\ell(u) < \ell(v)$.  We say that $\{ u,v \}$ is a {\em Bruhat
  edge} if either $u \rightarrow v$ or $v \rightarrow u$, and denote
this by $u \leftrightarrow v$.

The {\em Bruhat order} on $\symm_n$ is the partial order, which we
denote by $\leq$, which is the transitive closure of
$B(\symm_n)$. Unless explicitly stated, we always regard $\symm_n$ as
partially ordered by the Bruhat order. We follow \cite[Chapter
3]{StaEC1} for notation and terminology concerning posets.  The
following characterization of Bruhat order covering relations in the
symmetric groups is well known (see, e.g., \cite[Lemma 2.1.4]{BB}) and
will be repeatedly used in the sequel.
\begin{proposition}\label{pro:Bruhat cover}
	Let $u,v \in \symm_n$. Then the following conditions are equivalent:
	\begin{itemize}
		\item $u$ is covered by $v$ in Bruhat order (written $u\lcov v$);
		\item there are $1 \leq i < j \leq n$ such that $v=u (i \, j)$, $u(i) < u(j)$,
		and $\{ k \in [n] : i<k<j, u(i)<u(k)<u(j) \}= \emptyset$.
	\end{itemize}
\end{proposition}
Recall that a permutation $u\in \symm_n$ is said to \emph{contain the
	pattern}~$\alpha=\alpha_1\cdots \alpha_k$  if there exist $1\leq
i_1<\dots<i_k\leq n$ such that  $u(i_1),\dots,u(i_k)$ are in
the same relative order as  $\alpha_1,\dots,\alpha_k$. 
A permutation $u\in \symm_n$ is said to \emph{avoid} the pattern~$\alpha$
if it does not contain the pattern $\alpha$.  
We denote with $\Av_n(\alpha)=\{u\in \symm_n : u \mbox{ avoids
}\alpha\}$  the set of permutations of degree $n$ avoiding~$\alpha$.

\medskip

We graph $w = w(1)\cdots w(n) \in \symm_n$ using matrix coordinates:
the point $(i,w(i))$ appears in the $i$th row from the top of the grid
and the $w(i)$th column from the left. We let $G(w)$ denote the graph
of $w$.

\begin{example}
The permutation $41325$ is graphed in Figure~\ref{fig:G(41325)}. Related objects appear in Figures~\ref{fig:D(41325)} and~\ref{fig:Do(41325)}.
\end{example}

\begin{figure}[htbp]
$$\begin{tikzpicture}[scale=.5]
\draw[step=1.0,gray,very thin,xshift=-0.5cm,yshift=-0.5cm] (1,1) grid (6,6);
\foreach \x in {(4,5), (1,4), (3,3), (2,2), (5,1)} {\fill \x circle (4pt);}
\end{tikzpicture}$$
\caption{$G(41325)$.}\label{fig:G(41325)}
\end{figure}

The \emph{diagram} $D(w)$ of a permutation $w$ is
$$D(w) := \{ (i,j) \in [n]^2 : j< w(i), w^{-1}(j)>i \}.$$ 
The diagram can be seen by drawing lines to the south (\emph{legs})
and to the east (\emph{arms}) of each point $(i,w(i)) \in G(w)$, and
keeping the empty boxes that remain (see Figure~\ref{fig:D(41325)}).
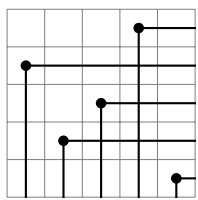
\begin{figure}[htbp]
$$\begin{tikzpicture}[scale=.5]
\draw[step=1.0,gray,very thin,xshift=-0.5cm,yshift=-0.5cm] (1,1) grid (6,6);
\begin{scope}
\clip (.5,.5) rectangle (5.5,5.5);
\foreach \x in {(4,5), (1,4), (3,3), (2,2), (5,1)}
  {\fill \x circle (4pt);
  \draw \x[thick] --++(0,-6);
  \draw \x[thick] --++(6,0);}
\end{scope}
\end{tikzpicture}$$
\caption{$D(41325)$ consists of the four empty boxes.}\label{fig:D(41325)}
\end{figure}

\noindent
\begin{definition}\label{def odd diag}
  The \emph{odd diagram} of a permutation $w$, as defined
  in~\cite{BC/20}, is the subset of $D(w)$ defined by
$$
\odd(w) := \{ (i,j) \in D(w) : \, i \not \equiv w^{-1}(j) \pmod{2}  \}.
$$
\end{definition}
We will often mark the elements of $D_o(w)$ by stars $*$, and refer to them as such.
The odd diagram of $41325 \in \symm_5$ is depicted in Figure~\ref{fig:Do(41325)}.
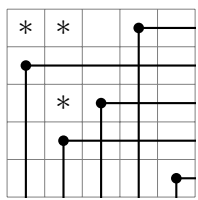
\begin{figure}[htbp]
$$\begin{tikzpicture}[scale=.5]
\draw[step=1.0,gray,very thin,xshift=-0.5cm,yshift=-0.5cm] (1,1) grid (6,6);
\begin{scope}
\clip (.5,.5) rectangle (5.5,5.5);
\foreach \x in {(4,5), (1,4), (3,3), (2,2), (5,1)}
  {\fill \x circle (4pt);
  \draw[thick] \x --++(0,-6);
  \draw[thick] \x --++(6,0);}
\end{scope}
\foreach \x in {(1,5),(2,3),(2,5)} {\draw \x node {$*$};}
\end{tikzpicture}$$
\caption{$\odd(41325)$ consists of the three boxes that are marked by
  $*$s.}\label{fig:Do(41325)}
\end{figure}

\medskip


\section{Injectivity and permutation patterns}\label{sec:injectivity}


Permutations in $\symm_n$ can be partitioned by their odd diagrams,
and we will indicate that two permutations are in the same odd diagram class by
$\sim$; that is, we write $v \sim w$ if $\odd(v) = \odd(w)$.  We begin
by noting that if two permutations have the same odd diagrams, then
the leftmost columns of their graphs must be the same.

\begin{lemma}\label{lem:leftmost column}
If $v \sim w$ then $v^{-1}(1) = w^{-1}(1)$.
\end{lemma}

\begin{proof}
  If the leftmost column of $\odd(v)$ is empty, then clearly $v(1)=1$
  and the lemma follows. More generally, since we are looking at the
  leftmost column of the (odd) diagrams, no boxes are eliminated by
  arms. Therefore the lowest star in $\odd(v) = \odd(w)$ sits directly
  north of the leftmost point in the graphs of $v$ and $w$.
\end{proof}

This enables us to prove something quite useful about permutations $v \sim w$.

\begin{theorem}\label{thm:equal odd means patterns}
If $v \sim w$ for $v \neq w$, then $v$ has a $213$-pattern and $w$ has a $312$-pattern,
 or conversely.
\end{theorem}

\begin{proof}
  Let $k+1$ be minimal such that $v^{-1}(k+1) \neq w^{-1}(k+1)$. By
  Lemma~\ref{lem:leftmost column}, $k \ge 1$.

To ease notation, set
\begin{align*}
a &:= v^{-1}(k) = w^{-1}(k),\\
b &:= v^{-1}(k+1), \text{ and}\\
c &:= w^{-1}(k+1).
\end{align*}
So $(a,k) \in G(v) \cap G(w)$, while $(b,k+1) \in G(v)$ and
$(c,k+1) \in G(w)$. Without loss of generality, assume that $b < c$.

Because $\odd(v) = \odd(w)$, the point $(c-1,k+1)$ must lie in the arm
of $(c-1,d) \in G(w)$ for some $d \le k$. By minimality of $k$, we
have that $(c-1,d) \in G(v)$, as well. Again by minimality of $k$, it
must also be the case that
$$v(c) > k+1 \text{ \ and \ } w(b) > k+1.$$
From this we find the desired patterns.
\end{proof}

We illustrate the proof of Theorem~\ref{thm:equal odd means patterns}
in Figure~\ref{fig:equal odd means patterns}. Note that the point
$(a,k) \in G(v) \cap G(w)$ is not needed for either of the patterns
unless $d = k$.
\begin{figure}[t]
\begin{tikzpicture}[scale=.5]
\draw[step=1.0,gray,very thin,xshift=-0.5cm,yshift=-0.5cm] (.5,.5) grid (8.5,8.5);
\foreach \x in {(1,2),(4,7),(7,1)} {\fill \x circle (4pt);}
\draw (.5,2) node[left] {$(c-1,d)$};
\draw (7,0) node[below] {$(c,v(c))$};
\draw (4,8) node[above] {$(b,k+1)$};
\draw (4,-2) node {$G(v)$};
\draw[thick, dashed] (3.5,0) -- (3.5,8);
\end{tikzpicture}
\hspace{.5in}
\begin{tikzpicture}[scale=.5]
\draw[step=1.0,gray,very thin,xshift=-0.5cm,yshift=-0.5cm] (.5,.5) grid (8.5,8.5);
\foreach \x in {(1,2),(6,7),(4,1)} {\fill \x circle (4pt);}
\draw (.5,2) node[left] {$(c-1,d)$};
\draw (4,0) node[below] {$(c,k+1)$};
\draw (6,8) node[above] {$(b,w(b))$};
\draw (4,-2) node {$G(w)$};
\draw[thick, dashed] (3.5,0) -- (3.5,8);
\end{tikzpicture}
\caption{The graphs of the permutations $v$ and $w$ as described in
  the proof of Theorem~\ref{thm:equal odd means patterns}, showing a
  $213$-pattern in $v$ and a $312$-pattern in $w$. The graphs are
  identical to the left of the dashed line.}\label{fig:equal odd means
  patterns}
\end{figure}
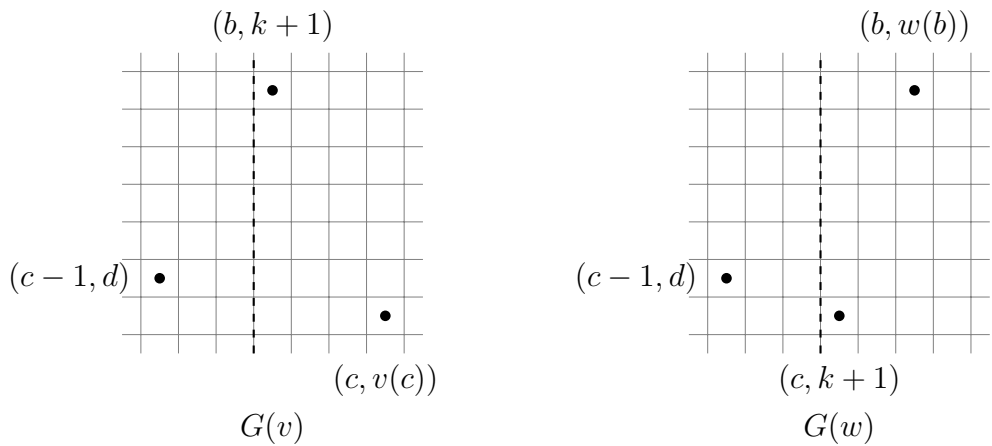

A \emph{vincular} permutation pattern is one in which pairs of letters
may be required to appear consecutively. For example, the permutation
$41325$ contains the classical pattern $321$ (demonstrated by the
substring $432$), but it does not contain the vincular pattern
$\ub{32}1$ because there is no occurrence of $321$ in which the first
two letters are consecutive.  In this language of vincular patterns,
then, Theorem~\ref{thm:equal odd means patterns} actually shows a
slightly stronger result.

\begin{corollary}\label{cor:theorem fallout}
  Suppose that $v \sim w$ for $v \neq w$. Then, without loss of
  generality, $v$ has a $2\ub{13}$-pattern and $w$ has a
  $3\ub{12}$-pattern such that
\begin{itemize}
\item these two patterns occupy the same positions in $v$ and $w$,
\item the value of the ``$2$'' is the same in each pattern, and
\item $G(v)$ and $G(w)$ coincide at all points $(a,b)$, for $b$ less
  than that shared value of ``$2$.'' \end{itemize}
\end{corollary}

Theorem~\ref{thm:equal odd means patterns} also shows how the map
$w \mapsto \odd(w)$ behaves on pattern classes, suggesting further
significance to the patterns $213$ and $312$ that appear in the
statement of the theorem, and verifying the first part of Conjecture
6.1 of \cite{BC/20}.

\begin{corollary}\label{cor:injective}
\begin{enumerate}
\item[(a)] The map $w\mapsto \odd(w)$ is injective on
  $\avoid_n(213)$. That is, if $v \neq w$ both avoid $213$, then
  $\odd(v) \neq \odd(w)$. Moreover, we can replace $213$ by $2\ub{13}$
  in both statements.
\item[(b)] The map $w\mapsto \odd(w)$ is injective on
  $\avoid_n(312)$. That is, if $v \neq w$ both avoid $312$, then
  $\odd(v) \neq \odd(w)$. Moreover, we can replace $312$ by $3\ub{12}$
  in both statements.
\item[(c)] For all other $p$ with $|\avoid_n(p)| > 1$, the map
  $w\mapsto \odd(w)$ is not injective on $\avoid_n(p)$.
\end{enumerate}
\end{corollary}
\begin{proof}
  Parts (a) and (b) are clear.  To prove (c), suppose first that $p$
  has one of the following three forms: $123\cdots k$ (with $k>2$), or
  $2134\cdots k$ (with $k>3$), or $3124\cdots k$ (with
  $k>3$). Then consider the permutations $v = n(n-1)\cdots 54213$ and
  $w = n(n-1)\cdots 54312$.  These both avoid $p$, and they have the
  same odd diagrams.

  Now, suppose that $p$ does not have one of those three forms and
  consider the permutations $v = 213456\cdots (n-1)n$ and
  $w = 312456\cdots(n-1)n$.  The only patterns these permutations
  contain are the three considered above. The two permutations $v$ and
  $w$ have the same odd diagram and, because $p$ is not one of the
  three patterns listed above, they both avoid $p$.
\end{proof}

Let $o_n:=|\{D_o(v):v\in \symm_n\}|$ denote the number of distinct odd
diagrams of permutations in $\symm_n$. It follows from
\cite[Propositions 1 and 3]{Claesson/01} that the number of
permutations in $\Av_n(3\ub{12})$ is given by the $n$th Bell number
$B_n$.  Thus the previous result gives a lower bound for the number of
odd diagrams in degree $n$.

\begin{corollary}
  Let $n\in \mathbb N$. Then $o_n\geq B_n$, where $B_n$ is the $n$th Bell number.
  \end{corollary}
  The first values of the sequence $\{o_n\}_{n\in \mathbb N}$ are:
  $1$, $2$, $5$, $17$, $ 70$, $351$, $2041$, $13732$, $103873$,
  $882213$ (cf.\ also \cite[A335926]{OEIS}).

  
  \section{Legal moves}\label{sec:legal moves}


\begin{definition}
  Recall from Section~\ref{sec:notation} that two permutations
  $v \neq \ol{v}$ are connected by a Bruhat edge if they agree on all
  but two values. A particular Bruhat edge $v \leftrightarrow \ol{v}$
  is \emph{legal} if $v \sim \ol{v}$. If $\ol{v}=vt\sim v$ we will
  also sometimes say that the transposition $t$ is \emph{legal for}
  $v$.
\end{definition}

Because a transposition only changes the positions of two values in
the permutation, legality depends only on the non-overlapping portions
of the points, arms, and legs that were affected by that
transposition.

\begin{lemma}\label{lem:only check rectangle}
  Consider a Bruhat edge $v\leftrightarrow \ol{v}$ where the points
  that move are as indicated in red and blue in Figure~\ref{fig:legal
    moves}. The Bruhat edge is legal if and only if none of the boxes
  that include only red or only blue in Figure~\ref{fig:legal moves}
  belong to $\odd(v)$ or to $\odd(\ol{v})$.
\end{lemma}

\begin{proof}
  Whether or not any of the other points is included in an odd diagram
  is not impacted by the points being swapped.
\end{proof}

\begin{figure}[htbp]
\begin{tikzpicture}[scale=.5]
\draw[step=1.0,gray,very thin,xshift=-0.5cm,yshift=-0.5cm] (.5,.5) grid (8.5,8.5);
\begin{scope}
\clip (0,0) rectangle (8,8);
\draw[blue,thick] (6,7) -- (2,7) -- (2,1);
\draw[red,thick] (6,7) -- (6,1) -- (2,1);
\foreach \x in {(6,1),(6,7)}{\draw[thick] \x --++(10,0);}
\foreach \x in {(2,1),(6,1)}{\draw[thick] \x --++(0,-10);}
\foreach \x in {(2,7),(6,1)}{\fill[blue] \x circle (4pt);}
\foreach \x in {(2,1),(6,7)} {\fill[red] \x circle (4pt);}
\end{scope}
\end{tikzpicture}
\caption{In a Bruhat edge $v\leftrightarrow \ol{v}$, the black arms
  and legs will arise for both permutations, whereas the red points
  and segments appear for only one permutation, and the blue points
  and segments only appear for the other.} \label{fig:legal moves}
\end{figure}
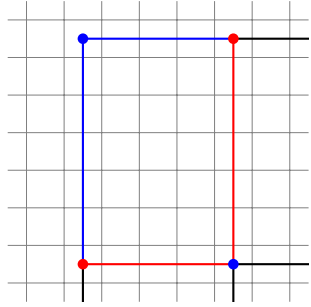

The following result is fundamental in all that follows.

\begin{theorem}\label{thm:legal swaps}
  Let $v$ be a permutation and $\ol{v}:=v \, (ij)$ for $i < j$. Set
  $m := \min\{v(i),v(j)\}$ and $M := \max\{v(i),v(j)\}$. The Bruhat
  edge $v \leftrightarrow \ol{v}$ is legal if and only if the
  following requirements are met:
\begin{enumerate}
\item[(R1)] $i$ and $j$ have the same parity,
\item[(R2)] $v(x) < m$ for all $x \in \{i+1,i+3,\ldots,j-1\}$, and
\item[(R3)] $v(y) \not\in [m,M]$ for all $y \in \{j+1,j+3,\ldots\}$.
\end{enumerate}
\end{theorem}

\begin{proof}
  As suggested by Lemma~\ref{lem:only check rectangle}, we prove the
  result by checking three things: the square marked by the blue point
  in the upper left of Figure~\ref{fig:legal moves}, the squares
  marked by vertical colored segments in that figure, and the squares
  marked by horizontal colored segments in that figure.
\begin{itemize}
\item The upper-left square marked by the blue point will be in one
  (but not both) of the odd diagrams for $v$ and $\ol{v}$ if and only
  if (R1) is not met.
\item For the vertical colored segments to contribute no elements to
  either odd diagram, all boxes in those columns whose heights differ
  in parity to $i$ (or $j$, by (R1)) must lie in the arms of points
  that appear to the left of $m$ in the graph. This is equivalent to
  (R2).
\item For the horizontal colored segments to contribute no elements to
  either odd diagram, there can be no boxes in those rows that are an
  odd distance above points in the graph. In light of (R2), this is
  equivalent to (R3).
\end{itemize}
Thus $v \sim\ol{v}$ if and only if (R1), (R2), and (R3) are met.
\end{proof}

One impact of Theorem~\ref{thm:legal swaps} is that if $v \sim \ol{v}$
is a legal Bruhat edge, then $v$ and $\ol{v}$ differ by the change of
a particular pattern.

\begin{definition}
  Let $v \leftrightarrow \ol{v}$ be a (not necessarily legal) Bruhat
  edge. If the points in which $v$ and $\ol{v}$ differ form the
  endpoints of a $213$-pattern in one of the permutations and a
  $312$-pattern in the other, then we call this a \emph{pattern swap},
  written $v \swap \ol{v}$. The \emph{type of the pattern swap for
    $v$} is the pattern ($213$ or $312$) in $v$ that gets changed to
  make $\ol{v}$.
\end{definition}
\noindent
So, for example, $5431627 \leftrightarrow 3451627$ is not a pattern swap, while
$5431627 \leftrightarrow 5431726$ is a pattern swap of type $213$ for $5431627$ 
(and of type $312$ for $5431726$).

\begin{corollary}
If a Bruhat edge is legal, then it must be a pattern swap.
\end{corollary}

\begin{proof}
  Suppose that $v \leftrightarrow \ol{v}$ is legal. By
  Theorem~\ref{thm:legal swaps}, (R1) and (R2) must be
  satisfied. Maintaining the notation from the proof of that theorem,
  (R1) says that $j \ge i+2$, and (R2) says that
  $v(j-1) = \ol{v}(j-1) < m$. Thus, the values in positions
  $\{i,j-1,j\}$ form a $213$-pattern in one of the permutations, and a
  $312$-pattern in the other, so $v \swap \ol{v}$.
\end{proof}

A pattern swap that results from a legal Bruhat edge will
correspondingly be called a \emph{legal} pattern swap. A pattern swap
$v \swap \ol{v}$ for which $\odd(v) \neq \odd(\ol{v})$ is an
\emph{illegal} pattern swap.

\vspace{.5em}

The properties discussed in Corollary~\ref{cor:theorem fallout} are
highly localized, and the permutations might differ substantially
otherwise. However, we can use these results to find much closer
``neighbors'' to a permutation, staying within the same class.

\begin{definition}\label{defn:intermediate permutations}
  Consider $v\sim w$ with $v \neq w$. Maintaining notation from
  Theorem~\ref{thm:equal odd means patterns}, let
  $\big\{(b,k+1), (c-1, d), (c,v(c))\big\}$ be the $213$-pattern found
  in $v$, and let $\big\{(b,w(b)), (c-1,d), (c,k+1)\big\}$ be the
  corresponding $312$-pattern found in $w$. Define permutations
  $\ol{v}_w:= v \, (b\,c)$ and $\ol{w}_v := w \, (b\,c)$.  In other
  words, $v \swap \ol{v}_w$ via the $213$-pattern found in $v$ during
  Theorem~\ref{thm:equal odd means patterns}, and $w \swap \ol{w}_v$
  via the $312$-pattern found in $w$.
\end{definition}

\begin{example}
  Consider $v = 5431627$ and $w = 7461325$. Then $\ol{v}_w =
  5461327$. Repeating the process with $u := \ol{v}_w$ and $w$
  produces $\ol{u}_w = 7461325 = w$.
\end{example}

\begin{remark}\label{rmk:intermediate lengths}
  Notice how the (classical) lengths of $v$ and $w$ compare to those of
the permutations described in Definition~\ref{defn:intermediate
  permutations}:
\begin{equation}
\ell(v) < \ell(\ol{v}_w) \text{ \ \ and \ \ } \ell(w) > \ell(\ol{w}_v).
\end{equation}\end{remark}
These new permutations act as intermediaries, allowing us to travel
between any two permutations in the same odd diagram class.

\begin{theorem}\label{thm:intermediaries}
Suppose that $v \sim w$ with $v \neq w$. Then $v \sim \ol{v}_w$ and $w \sim \ol{w}_v$. 
\end{theorem}

\begin{proof} We show that $v \sim \ol{v}_w$.  Maintain the notation
  of Theorem~\ref{thm:equal odd means patterns} and Definition
  \ref{defn:intermediate permutations}. Since $v \sim w$, we have that
  $b \equiv c \pmod{2}$ (else $(b,k+1) \in D_o(w) \setminus D_o(v)$).
  Also, $v(x)<k+1$ if $x \in \{ b+1,b+3, \ldots , c-1\}$ (otherwise
  $w(x)>k+1$ by the minimality of $k$, so
  $(x,k+1) \in D_o(w) \setminus D_o(v)$).  Finally,
  $v(x) \notin [k+2,v(c)-1]$ if $x \in \{ c+1,c+3, \ldots \}$ (else
  $(c,v(x)) \in D_o(v)\setminus D_o(w)$). Hence, by Theorem
  \ref{thm:legal swaps}, the Bruhat edge $v \leftrightarrow \ol{v}_w $
  is legal and therefore $v \sim \ol{v}_w$. The proof that
  $w \sim \ol{w}_v$ is similar.
\end{proof}

Note an interesting consequence of Definition~\ref{defn:intermediate
  permutations} and Theorem~\ref{thm:intermediaries}.

\begin{corollary}\label{cor:stepping closer}
  Consider permutations $v \sim w$ that agree for the first $k$, but
  not $k+1$, values (that is, $v^{-1}(i) = w^{-1}(i)$ for all
  $i \le k$, but not for $i = k+1$). Then the permutations $\ol{v}_w $
  and $ w$ agree for the first $k+1$ values, at least, as do
  $\ol{w}_v $ and $ v$.
\end{corollary}

From these results we see that the class of permutations with a given
odd diagram is, in a sense, connected.

\begin{definition}
Let $D$ be an odd diagram, and consider the set
$$\perm_n(D) := \{w \in \symm_n : \odd(w) = D\}.$$
Define the \emph{class graph} $G_{D,n}$ to have vertex set
$\perm_n(D)$, and an edge between vertices $v$ and $v'$ if
$v' = \ol{v}_w$ for some permutation $w \in \perm_n(D)$.
\end{definition}

\begin{corollary}\label{cor:connected graph}
The class graph $G_{D,n}$ is connected.
\end{corollary}

We conclude this section with the following result, which resolves the
length conjecture \cite[Conjecture 6.1]{BC/20} and concludes the proof
of Theorem~\ref{thm:avoid conjecture}.

\begin{corollary}\label{cor:length conjecture}\
\begin{enumerate}
\item[(a)] If $v \in \perm_n(D)$ is $312$-avoiding, then $v \leq u$
  for all $u \in \perm_n(D)$.
\item[(b)] If $w \in \perm_n(D)$ is $213$-avoiding, then $w \geq u$
  for all $u \in \perm_n(D)$.
\end{enumerate}
\end{corollary}

\begin{proof}
  Let $u$ be a minimal (with respect to the Bruhat order) element of
  $\{ z \in \Perm_n(D) : v \not \leq z \}$. Then, by
  Theorem~\ref{thm:equal odd means patterns} and our hypothesis, $v$
  has a $213$-pattern and $u$ has a $312$-pattern.  Thus
  $\ol{u}_v < u$ (notation as in Definition \ref{defn:intermediate
    permutations}) which contradicts the minimality of $u$ since, by
  Theorem \ref{thm:intermediaries}, $u \sim \ol{u}_v$. The proof of
  (b) is analogous.
\end{proof}

\medskip


\section{Consequences of legality}\label{sec:length}


While legality implies that a Bruhat edge is a pattern swap, not all
pattern swaps are legal. In fact, as we show below, the potential for
illegal pattern swaps is, in a sense, persistent in a class.

\begin{definition}
  If $v$ has an illegal pattern swap of type $\pi$, then we say that
  the pattern represented by $\pi$ is an \emph{illegal pattern} in
  $v$.
\end{definition}

In other words, an illegal pattern is a $213$- or $312$-pattern in $v$
for which the permutation $v'$ obtained by swapping the left and right
letters in the pattern, does not have the same odd diagram as $v$.

\begin{theorem}\label{thm:illegal pattern swaps persist}
  Suppose that $v\in \symm_n$ has an illegal pattern of type
  $\pi$. Then all $w \sim v$ also have illegal patterns of type $\pi$.
\end{theorem}

\begin{proof}
  Set $D := \odd(v)$. We prove this result recursively, showing that
  if $v$ has this property, then every neighbor $\ol{v}_w$ of $v$ in
  the graph $G_{D,n}$ defined above also has this property. By
  Corollary~\ref{cor:connected graph}, this will prove the result.

  Let $1 \leq i < h < j \leq n$ be the positions of the illegal
  pattern $\pi$ in $v$. Set $m := \min\{v(i),v(j)\}$ and
  $M := \max\{v(i),v(j)\}$.

  We first check whether there is a $\ol{v}_w$ with no illegal
  patterns of type $\pi$, in which $v(i) = \ol{v}_w(i)$ and
  $v(j) = \ol{v}_w(j)$. Suppose that there is such a $\ol{v}_w$.  Note
  that if $\ol{v}_w\swap v$ legally moved $(h,v(h))$, then for at
  least one of $h' \in \{h\pm1\}$, we would have $v(h') < v(h)$ by
  (R2). Define the position
$$h^* := v^{-1}\Big(\min\big\{v(h-1),v(h),v(h+1)\big\}\Big),$$
and the pattern
$$
\pi ' := \Big\{ \big(i,\ol{v}_w(i)\big), \big(h^*,\ol{v}_w(h^*)\big), \big(j,\ol{v}_w(j)\big)\Big\}
$$
in $\ol{v}_w$. If $\pi$ is illegal because $i \not\equiv j \pmod{2}$,
then $\pi'$ is illegal in $\ol{v}_w$. We may therefore assume that
$i \equiv j \pmod{2}$.
Now suppose that there exists $x \in \{ i+1,i+3, \ldots , j-1 \}$ such
that $v(x)>m$. If $\ol{v}_w(x) = v(x)$, then $\pi'$ is illegal in
$\ol{v}_w$. Otherwise $\ol{v}_w = v (x \, z)$ for some
$z \not\in \{i,j,x\}$ with $v(z) < m$. But then either $(i,v(i))$ or
$(j,v(j))$ would have violated (R2) for the supposedly legal
$v \swap v(x\,z)$. Therefore we can assume that there is no such $x$.
Finally, suppose that there exists $y \in \{j+1,j+3,\ldots\}$ such
that $v(y) \in [m,M]$. If $\ol{v}_w(y) = v(y)$, then $\pi'$ is illegal
in $\ol{v}_w$. Otherwise, $\ol{v}_w = v (y \, z)$ for some
$z \not\in \{i,j,y\}$ with $v(z) \not\in [m,M]$. To be legal, we must
have $z \equiv y \pmod{2}$. If $z > j$, then $\pi'$ will be illegal
due to $(z,\ol{v}_w(z))$. On the other hand, if $z < j$, then
multiplying by $(y \, z)$ would have failed (R2) because of the point
$(v^{-1}(M),M)$, and hence would not have been a legal move.
Therefore there is no such $\ol{v}_w$.

Now consider $\ol{v}_w$ in which, without loss of generality,
$\ol{v}_w = v(i \, i')$ for some $i' \neq i$, and in particular
$i \neq h' \neq j$ because $v(h') < m$. Note that $i' \neq j$ because
the move $\ol{v}_w \leftrightarrow v$ is legal; on the other hand, the
case of $i'$ equalling $h$ is not excluded.

Suppose that (R1) is not met in $v$. To fix this, the point at height
$i$ must move legally to form $\ol{v}_w$. By (R1), $i$ and $i'$ have
the same parity.  If $i = h'$, and $t$ is the position of the middle
value in the (necessarily) $312$-pattern swapped to form $\ol{v}_w$,
then
$$\Big\{ \big(i,\ol{v}_w(i)\big), \big(t,\ol{v}_w(t)\big), \big(j,\ol{v}_w(j)\big)\Big\}$$
will be an illegal pattern in $\ol{v}_w$ failing (R2), so we can assume $i' \neq h$. 
Thus, in all but one case, we have that
$$\Big\{ \big(I,\ol{v}_w(I)\big), \big(h,\ol{v}_w(h)\big), \big(j,\ol{v}_w(j)\big)\Big\}$$
is an illegal pattern of type $\pi$ in $\ol{v}_w$, for at least one
$I \in \{i,i'\}$, failing (R1). The only case where this might not
hold is when $i < h < i'$ and $v(i') < v(h)$. For $v \swap \ol{v}_w$
to have been legal, we must have had $v(i+1) < \min\{v(i),v(i')\}$,
and thus
$$\Big\{ \big(i,\ol{v}_w(i)\big), \big(i+1,\ol{v}_w(i+1)\big), \big(j,\ol{v}_w(j)\big)\Big\}$$
is an illegal pattern of type $\pi$ in $\ol{v}_w$, again failing
(R1). Thus we may assume that (R1) is satisfied, and hence
$i' \equiv i \equiv j \pmod 2$.

Now suppose that (R2) is not met in $v$. Let
$x \in \{i+1,i+3,\ldots, j-1\}$ be maximal for which $v(x) \ge m$. To
fix this in $\ol{v}_w$, the point at height $i$ would have to move
legally downward, using $i' \ge x$. In fact, because $i$ and $i'$ have
the same parity, we must have $x \in \{i+1, i+3, \ldots, i' - 1\}$,
and $v(x) < \min\{v(i),v(i')\}$ to make this a legal move not
violating (R2). Therefore, it must be that $m = v(j)$, and
$$\left\{(i,\ol{v}_w(i)), (h,v(h)), (j,v(j))\right\}$$
is an illegal pattern of type $\pi$ in $\ol{v}_w$, again failing (R2)
with this value of $x$.

Finally, suppose that (R3) is not met in $v$. In any legal Bruhat edge
from $v$, we find that
$$\left\{(i,\ol{v}_w(i)), (j-1,v(j-1)), (j,v(j))\right\}$$
is an illegal pattern of type $\pi$ in $\ol{v}_w$, again failing (R3).

The case $\ol{v}_w = v(j \, j')$ can be addressed by the same arguments.
\end{proof}

\medskip


\section{Odd diagram classes  are  Bruhat intervals}\label{sec:bruhat}


This section is devoted to proving the following result.

 \begin{theorem}\label{thm:bruhat}
   Let $D\subset [n]^2$ be an odd diagram. Then  $\Perm_n(D)=[u,v]$ for some $u,v\in \symm_n$.
 \end{theorem}

 We prove Theorem~\ref{thm:bruhat} in three main steps. The first is
 to show that given an odd diagram class, there exist one
 Bruhat-minimal and one Bruhat-maximal element in the class. Next, we
 will show that the class contains a maximal chain between those two
 extreme elements. Finally, we will use the ``flip'' operation of
 \cite{BilBre} to complete the argument.

 For the rest of the paper, assume that $|\perm_n(D)| \ge 2$. We start
 by showing that within a class, each value $k \in [n]$ can only
 appear in positions having the same parity.
 
\begin{lemma}\label{lem:basic}
  Let $u,v\in \symm_n$ with $u\sim v$. Then
  $u^{-1}(k)\equiv v^{-1}(k)\pmod 2$ for all $k\in [1,n]$.
    \end{lemma}
    
    \begin{proof} 
      Set $D:=D_o(u)=D_o(v)$. Since $G_{D,n}$ is connected by
      Corollary~\ref{cor:connected graph}, we may assume that $u$ and
      $v$ are connected by an edge. Hence there is $w\in\Perm_n(D)$,
      $w\neq v$ such that $u=\overline{v}_w$. This means that
      $u=v(b\,c)$ for some $b,c \in [n]$, and $u \leftrightarrow v$ is
      a legal Bruhat edge. Hence, by Theorem \ref{thm:legal swaps},
      $b \equiv c \pmod{2}$.

Thus $u^{-1}(i)\equiv v^{-1}(i)\pmod 2$ for all $i\in [n]$, as desired.
 \end{proof}
 
 Note that the lemma implies, in particular, that we can talk about
 ``admissible parity'' of a column of the graph of a permutation in a
 class.

\begin{definition}\label{defn:admissible parity}
  Fix an odd diagram $D$. Let the $k$th column be labeled
  $\varepsilon_k \in \{0,1\}$ according to the parity of $w^{-1}(k)$
  for some (every, by Lemma~\ref{lem:basic}) permutation
  $w \in \perm_n(D)$. An arbitrary permutation $v$ has
  \emph{admissible parity} if $v^{-1}(k) \equiv \varepsilon_k \pmod 2$
  for all $k$.
\end{definition}
    
We exploit this notion in the next theorem to construct, starting from
any permutation, the unique minimal element of its odd diagram
class. The idea is to define a ``smallest'' permutation with the given
odd diagram.
    
\begin{theorem}\label{thm:bottomchain}
  Fix an odd diagram $D$. There exists $u \in \perm_n(D)$ such that,
  for all $w \in \perm_n(D)$, $u \le w$ in Bruhat order.
\end{theorem}   
\begin{proof}
  Based on the odd diagram $D$, we will construct the permutation $u$
  by placing dots (points in the graph of $u$) in each column, from
  left to right. For each new column, we place a dot in the highest
  empty cell which does not have a star below it or to the right, and
  which has no dots already placed to its left, and which has
  admissible parity.

  Consider an arbitrary $w \in \perm_n(D)$.  In order to show that the
  desired $u$ exists, we will construct a sequence of permutations
  $w = w_1, w_2,\dots,w_n$ such that
  \begin{enumerate}
  \item[(i)] $w_{k+1}\leq w_k$ in Bruhat order,
  \item[(ii)] $w^{-1}_{k+1}(i)=w^{-1}_{k}(i)$ for all $i$, except for
    (at most) two values of $i\in[k+1,n]$, and
  \item[(iii)] $w_{k}\sim w_{k+1}$.
  \end{enumerate}
  Developing $w_2, w_3, \ldots, w_n$ will be based on the data of $D$,
  not of $w$. The only feature of $w$ which could be considered
  relevant is $w^{-1}(1)$, but Lemma~\ref{lem:leftmost column} means
  that this is forced by $D$ itself.
  
  Consider $k\ge1$ and assume $w_1,\ldots,w_k$ have already been
  defined.  Set $i_{k+1}=\max\{i\in[n]:(i,k+1)\in D\}$ if the latter
  set is non-empty, and $i_{k+1}=0$ otherwise. Let $\mathcal A_{k+1}$
  be the set of $r\in [n]$ such that:
  \begin{itemize}
    \item $r>i_{k+1}$,
  \item $r\equiv \varepsilon_{k+1}\pmod 2$,
  \item $(r,s)\notin D$ if $k+1<s\leq n$, and
  \item $r\notin \{w_k^{-1}(i) : i\in[k]\}$.
  \end{itemize}
  Informally, $\mathcal A_{k+1}$ is the set of admissible positions
  for an element in column $k+1$ of the graph of a permutation
  equivalent to $w_k$, and for which the values $1,\dots, k$ have the
  same positions that they had in $w_k$. In particular,
  $w_k^{-1}(k+1)\in\mathcal A_{k+1}$, so this set is non-empty.

  Let $b=\min \mathcal A_{k+1}$. If $b=w_k^{-1}(k+1)$ then
  $w_{k+1}:=w_k$.  Otherwise, set $c:=w_k^{-1}(k+1)$. By definition,
  $b<c$ and $b\equiv c \pmod 2$. We set $w_{k+1}:=w_k(b\,c)$. It is
  clear that $w_{k+1}$ satisfies properties (i) and (ii). To show that
  $w_{k+1}\sim w_{k}$ it is enough to show that $(b\,c)$ is a legal
  transposition for $w_k$, using Theorem~\ref{thm:legal swaps}.
  \begin{itemize}
  \item (R1) follows from the parity condition in $\mathcal A_{k+1}$. 
  \item (R2) translates to showing that $w_k(x)<w_k(c)$ for all
    $x\in\{b+1,b+3,\dots,c-1\}$. Indeed, if $w_k(x)>w_k(c)=k+1$ for
    such an $x$ then $(x,k+1)\in D$, which would contradict the
    maximality of $i_{k+1}$.
  \item (R3) translates to showing that $w_k(y)\notin[w_k(c),w_k(b)]$
    for all $y\in\{c+1,c+3\dots\}$. If, instead,
    $w_k(y)\in[w_k(c),w_k(b)]$ then $(b,w_k(y))\in D$, which would
    contradict $b\in\mathcal A_{k+1}$.
    \end{itemize}
Therefore $w_{k+1}\sim w_k$.

The last permutation of the sequence $u:=w_n$ is the minimal element
of $\perm_n(D)$. \end{proof}

An analogue of the above result can be used to construct the maximal
permutation in a class, as stated in the following.

\begin{theorem}\label{thm:topchain}
  Fix an odd diagram $D$. There exists $v \in \perm_n(D)$ such that,
  for all $w \in \perm_n(D)$, $v \ge w$ in Bruhat order.
\end{theorem}

\begin{proof}
  The proof follows similar lines to those of
  Theorem~\ref{thm:bottomchain}. Here the idea is to construct the
  maximal element by choosing, in every column, the largest admissible
  position. Keeping notation as in Theorem~\ref{thm:bottomchain}, we
  define the sequence of permutations $w_1,\dots,w_n$ such that
  \begin{enumerate}
  \item[(i)] $w_{k+1}\geq w_k$ in Bruhat order,
  \item[(ii)] $w^{-1}_{k+1}(i)=w^{-1}_{k}(i)$ for all $i$, except for
    (at most) two values of $i\in[k+1,n]$, and
  \item[(iii)] $w_{k}\sim w_{k+1}$.
  \end{enumerate}
  At each step $w_{k+1}$ is defined as in
  Theorem~\ref{thm:bottomchain} except here we take
  $b=\max \mathcal A_{k+1}$. Arguing as in the previous theorem shows
  that $v:=w_n$ is the maximal element of $\perm_n(D)$.
  \end{proof}

We demonstrate these results with an example.

\begin{example}
  Let $w=7461325\in \symm_7$. The graph $G(w)$ and odd diagram
  $D=D_o(w)$ are depicted in
  Figure~\ref{fig:construction}(\textsc{a}).
  Figure~\ref{fig:construction}(\textsc{b}) illustrates a step in the
  construction of the minimal element of $\Perm_7(D)$. The shaded grey
  cells are the admissible positions in the third column
  ($\mathcal A_3$ in the proofs of Theorems~\ref{thm:bottomchain}
  and~\ref{thm:topchain}).  Figure~\ref{fig:construction}(\textsc{c})
  depicts the graphs of the minimal (5431627, whose graph is
  represented by $\circ$) and maximal (7461523, whose graph is
  represented by $\square$) elements in the class. $\Perm_7(D)$ is a
  Bruhat interval of size 18 and rank~5.
      \end{example}

      \begin{figure}[t]
        \begin{subfigure}[t]{0.32\textwidth}\centering
\begin{tikzpicture}[scale=.5]
  \draw[step=1.0,gray,very thin,xshift=-0.5cm,yshift=-0.5cm] (1,1) grid (8,8);
\begin{scope}
\foreach \x in {(7,7), (4,6), (6,5), (1,4), (3,3), (2,2), (5,1)}
{\fill \x circle (4pt);}
\foreach \x in {(1,7), (2,7), (4,7), (3,6), (1,5), (2,5), (2,3)}
   {\draw \x node {$*$};}
\end{scope}
\end{tikzpicture}%
\caption{$G(w)$ and $D_o(w)$}\label{fig:G and D}  
\end{subfigure}
\begin{subfigure}[t]{0.32\textwidth}\centering
\begin{tikzpicture}[scale=.5]
  \draw[step=1.0,gray,very thin,xshift=-0.5cm,yshift=-0.5cm] (1,1) grid (8,8);
\begin{scope}
\foreach \x in { (1,4),  (2,2)}
{\draw[fill=none] \x circle (4pt);}
\foreach \x in {(1,7), (2,7), (4,7), (3,6), (1,5), (2,5), (2,3)}
{\draw \x node {$*$};}
\foreach \x in {(3,5), (3,3), (3,1)}
 {\draw[draw=gray,very thin,fill=black!20!,xshift=-0.5cm,yshift=-0.5cm,]  \x rectangle ++(1,1);}
\end{scope}
\end{tikzpicture}%
\caption{Admissible positions}\label{fig:admissible3}
\end{subfigure}
\begin{subfigure}[t]{0.32\textwidth}\centering
\begin{tikzpicture}[scale=.5]
  \draw[step=1.0,gray,very thin,xshift=-0.5cm,yshift=-0.5cm] (1,1) grid (8,8);
\begin{scope}
\foreach \x in { (1,4),  (2,2), (3,5), (4,6), (5,7), (6,3), (7,1)}
{\draw[fill=none] \x circle (4pt);}
\foreach \x in {(1,7), (2,7), (4,7), (3,6), (1,5), (2,5), (2,3)}
{\draw \x node {$*$};}
\foreach \x in { (1,4),  (2,2), (3,1), (4,6), (5,3), (6,5), (7,7)}
 {\draw[draw=black,fill=none,xshift=-5pt,yshift=-5pt,]  \x rectangle ++(10pt,10pt);}
\end{scope}
\end{tikzpicture}
\caption{Min ($\circ$) and max ($\square$)}\label{fig:minmax}
\end{subfigure}
\caption{Building the minimal and maximal elements of
  $\Perm_7(D_o(7461325))$.}\label{fig:construction}
\end{figure}
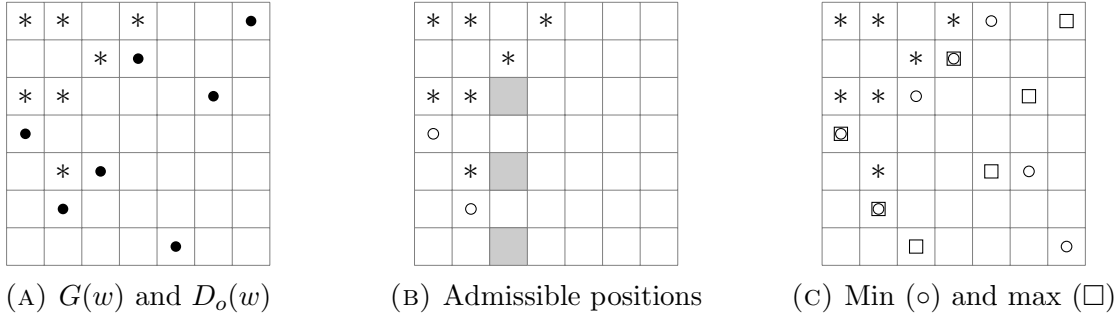

Theorems~\ref{thm:bottomchain} and \ref{thm:topchain} imply that the
minimal (``bottom'') and maximal (``top'') elements, $u$ and $v$, of
the class are unique and that $\Perm_n(D)\subseteq[u,v]$.

\vspace{.5em}
 
We now show that given an odd diagram class $\Perm_n(D)$ and its
bottom and top elements $u$ and $v$, there is a maximal chain (in
Bruhat order) from $u$ to $v$ of elements within the class. We start
by showing that given any permutation $w \neq v$, we can always find
an element of $\Perm_n(D)$ covering $w$. We will sometimes call such a
permutation a \emph{legal cover} of $w$.

 \begin{proposition}\label{pro:maxchain}
   Let $w\in \symm_n$ and $D=D_o(w)$. Let the bottom and top elements
   of $\Perm_n(D)$ be $u$ and $v$, respectively. Then there exists a
   maximal chain of elements in the class $\Perm_n(D)$ connecting $u$
   to $v$.
  \end{proposition}
Our proof will rely on the following.
\begin{lemma}\label{lem:legalcover}

  Let $v$ denote the top element of $\Perm_n(D)$ and assume $w\neq
  v$. Then there exists a transposition $t$ such that $w\lra wt$ is a
  legal Bruhat edge and $w\lcov wt$.
  \end{lemma}
  \begin{proof}
    Since $w\sim v$ and $w\neq v$, Definition~\ref{defn:intermediate
      permutations} and Theorem~\ref{thm:intermediaries} ensure that
    there exists $\ol{w}_v\sim w$ which is obtained from $w$ by
    applying a single (legal) transposition, say $r$. If $k$ is the
    minimum column index in which $w$ and $v$ differ, then $r=(b\,c)$,
    where $b=w^{-1}(k)$ and $c=v^{-1}(k)$. Since $(b\, c)$ is legal
    for $w$, $b\equiv c \pmod 2$. Since $v$ is the longest element of
    the class, $\ell(\ol{v}_w)<\ell(v)$ and it therefore follows from
    Remark~\ref{rmk:intermediate lengths} that $\ell(\ol{w}_v)>\ell(w)$ so $b<c$. If
    $\ell(\ol{w}_v)-\ell(w)=1$ we are done.

    If $\ell(\ol{w}_v)-\ell(w)>1$, that is
    $\ell(\ol{w}_v)-\ell(w)\geq 3$, then $w(c)=k+i$ for some
    $i\geq 2$, $c-b\geq 2$ and, by Proposition~\ref{pro:Bruhat cover},
    there exists $j\in[k+1,k+i-1]$ such that
    $d:=w^{-1}(j)\in [b+1,c-1]$. At the level of the graph $G(w)$ this
    means that the points $(b,w(b))$ and $(c,w(c))$ are sufficiently
    far away from each other and that there is at least one point of
    the graph inside the rectangle they determine.  Let
    $m:=\min\{d\in[n]:(d,w(d))\in[b+1,c-1]\times[k+1,k+i-1]\}$ so
    $(m,w(m))$ is the ``highest'' such dot. We claim that $t=(b\, m)$
    is a legal transposition for~$w$ such that $w\lcov wt$.  The
    transposition $(b\,c)$ is legal for $w$, so, by Theorem
    \ref{thm:legal swaps} $w(x)<w(b)$ for all
    $x\in\{b+1,b+3,\dots,c-1\}$. Since $w(b)<w(m)$ this implies that
    $m\equiv b\pmod 2$, showing that both (R1) and (R2) from
    Theorem~\ref{thm:legal swaps} hold. Finally, conditions (R2) and
    (R3) for $w\lra w(b\,c)$ imply that $w(y)<w(b)$ if
    $y \in \{ m+1,m+3, \ldots , c-1 \}$ and that
    $w(y) \notin [w(b),w(m)]$ if $y \in \{ c+1,c+3, \ldots \}$ so (R3)
    holds for $w\lra w(b\, m)$.

    Clearly, by minimality of $m$, the transposition $t$ is also such
    that $w\lcov wt$. This proves the result.\end{proof}

  \begin{proof}[Proof of Proposition~\ref{pro:maxchain}]
    Our result follows by repeated application of
    Lemma~\ref{lem:legalcover} noting that, since $wt \in \Perm_n(D)$,
    $wt \leq v$ by Theorem \ref{thm:topchain}.
  \end{proof}

  We now show that any saturated $3$-chain in an odd diagram class
  completes to a square in the same class.  Let $x, y, z \in \symm_n$
  be such that $x \lhd y \lhd z$. By~\cite[Lemma 2.7.3]{BB} the
  interval $[x, z]$ is isomorphic, as a poset, to a Boolean algebra of
  rank $2$. Thus, there exists a unique permutation $w \in \symm_n$
  such that $w \neq y$ and $x \lhd w \lhd z$. The proof of the
  following proposition relies on a case by case analysis that is
  based on an explicit description of $w$ that we now provide.

Let $y = x(a \, b)$ for some $a < b$ and $z = y(c \, d)$ for some $c < d$. We have
$x(a) < x(b)$ and $y(c) < y(d)$ and there are five cases:
\begin{itemize}
	\item $|\{a, b, c, d\}|=4$: In this case $w = x(c\,d) = z(a\,b)$.
	\item $a < b = c < d$: In this case we have $x(a) < x(d)$,
          $w = x(c\,d)$ if $x(b) < x(d)$ and $w = z(a\,b)$ if
          $x(b) > x(d)$.
	\item $a < b = d > c$: In this case we have $x(a) > x(c)$,
          $w = z(a \, b)$ if $c < a$ and $w = x(c\,d)$ if $a < c$.
	\item $b > a = c < d$: In this case we have $x(b) < x(d)$,
          $w = z(a\,b)$ if $d > b$ and $w = x(c\,d)$ if $d < b$.
	\item $b > a = d > c$: In this case we have $x(c) < x(b)$,
          $w = z(a\,b)$ if $x(c) > x(a)$ and $w = x(c\,d)$ if
          $x(c) < x(a)$.
\end{itemize}

  \begin{proposition}\label{pro:square}
    Let $D$ be an odd diagram, with $x\lcov y\lcov z$ all in
    $\Perm_n(D)$. Then $[x,z]\subseteq \Perm_n(D)$.
\end{proposition}

  \begin{proof}
    We keep the notation introduced before the statement of
    Proposition \ref{pro:square}. We will show that either $x \lra w$
    or $w \lra z$ is a legal Bruhat edge.  We need to consider several
    cases, depending on $|\{a,b,c,d\}|$, parities and the relative
    order of positions and values involved in the swaps.

    \vspace{2mm}
    
    \textbf{Case 1.} Suppose $|\{a,b,c,d\}|=4$, that is $(a\,b)$ and
    $(c\,d)$ commute. We claim that in this case $(c\,d)$ is a legal
    move for $x$ such that $w=x(c\,d)$ covers $x$.  Clearly, the
    parity condition (R1) from Theorem~\ref{thm:legal swaps} holds. By
    our assumption, $x(a)<x(b)$ and $x(c)<x(d)$. To show that
    $x\lra w$ is a legal move, we need to show that (R2) and (R3)
    hold.

\begin{itemize}
\item If $a\equiv c\pmod 2$ then $x$, $y$, $w$ and $z$ coincide in all
  positions $i\in[n]$ with $i\not\equiv a\pmod 2$. But these are the
  only values involved in the requirements for the legality of the
  relevant moves. So (R2) and (R3) for $y\lra z $ imply the analogous
  conditions for $x\lra w$ independently of the relative order of
  $a,b,c$ and $d$.

\item For $a\not\equiv c\pmod 2$, we will consider $a<c$ (all other
  cases work similarly).

    If $a<b<c<d$ then clearly (R2) and (R3) hold for $y \lra z $ if
    and only if they hold for  $x\lra w$, since positions and values
    involved in the swaps are the same.
    
    The case in which the two transpositions interlace, that is
    $a<c<b<d$, cannot occur. Indeed, for $(a\,b)$ to be legal for
    $x$, $x(c)<x(a)$ should hold, and for $(c\,d)$ to be legal for $y$, 
    $x(a)=y(b)<y(c)=x(c)$ should hold.

    Finally, suppose $a<c<d<b$. Our assumptions and condition (R2) for
    $x\lra y$ imply $x(c)<x(d)<x(a)<x(b)$. This in turn shows that
    both (R2) and (R3) hold for $x\lra w$ if and only if they hold
    for $y \lra z$. 
\end{itemize}
    
        \vspace{2mm}

    \textbf{Case 2.}  Suppose $|\{a,b,c,d\}|=3$. 
    The parity condition is clearly always satisfied in this case. 

\smallskip
   Suppose $a<b=c<d$. 
    By assumption, we have $x(a)<x(b)$ and $x(a)<x(d)$ (since $x\leq y\leq z$).

%

     \begin{itemize}
       
     \item 
    If $x(a)<x(b)<x(d)$, 
     then $w=x (c \, d)$ and (R2) and (R3)  for $y \lra z$  imply
    (R2)  and (R3) for $x\lra w$. This is because
    $\min\{x(a),x(d)\}=x(a)<x(b)=\min\{x(b),x(d)\}$ and $\max\{ x(b),x(d) \}=\max\{x(a),x(d)\}$. 
  \item 
    If $x(a)<x(d)<x(b)$ then $w=z(a\, b)$ and 
    we claim that $z$ is a legal cover of $w$, or, equivalently, that
    $x \lra x(a\,d)$ is legal.  We show the
    latter.  Condition (R2) for $x\lra y$ and for $y\lra z$ implies
    $x(i)<x(a)$ for all $i\in \{a+1,a+3,\dots,d-1\}$; that is, (R2)
    holds for $x\lra x(a\,d)$.  Similarly, condition~(R3) for $y\lra z$ 
    implies (R3) for  $x\lra x(a\,d)$. 
    \end{itemize}
 
Suppose now that  $a=c<b<d$. Then, since $x \leq y \leq z$, $x(a)<x(b)<x(d)$,
and $w=z(a \, b)$.
We claim that then $w \lra z$ is a legal move. Indeed, condition (R2) for 
 $z \lra z (a \, b)$ holds if and only if it holds for $x \lra y$. Furthermore,
 if $i \in \{ b+1, b+3, \ldots  \}$ then $z(i) \notin [x(a),x(b)]$ because 
 condition (R3) holds for $x \lra y$ and because $i \neq d$. While $z(i) \notin 
 [x(b)+1,x(d)]$ if $i < d$ because $y \lhd z$ and if $i>d$ because condition (R3) holds for 
 $y \lra z$. Hence condition (R3) holds for $z \lra z (a \, b)$.
 
Suppose now that  $a=c<d<b$. Then $x(a)<x(b)<x(d)$ and $w=x(c \, d)$. We claim that in this case
$x \lra w$ is a legal move. Indeed, condition (R2) for $x \lra x(c \, d)$
follows from condition (R2) for $x \lra y$. Let $i \in \{ d+1, d+3, \ldots  \}$,
so $i \neq b$. Then $x(i) \notin  [x(a),x(b)]$ if $i<b$ because $x \lhd y$ and if $i>b$ because 
condition (R3) holds for $x \lra y$. While $x(i) \notin 
[x(b)+1,x(d)]$ because condition (R3) holds for 
$y \lra z$. Hence condition (R3) holds for $x \lra w$.
 
 The proofs of the cases $a<c<b=d$, $c<a<b=d$ and $c<a=d<b$ are all 
 analogous, and are therefore omitted. 
 \qedhere
   
    \end{proof}

    \vspace{1em}

    Let $u,v \in \symm_n$ be such that $u \rightarrow v$ in the Bruhat
    graph, and set $\lambda(u,v):=v \, u^{-1} \in T$.  If
    $(x_0, \ldots , x_d) \in \symm_n^d$ is a saturated chain, then
    define
$$\lambda(x_0, \ldots , x_d):=(\lambda(x_0,x_1), \ldots ,\lambda(x_{d-1},x_d)) \in T^d.$$
Let $\preceq$ be the lexicographic order on $T$, so
$(1\,2) \preceq (1\,3) \preceq \cdots \preceq (1\,n) \preceq (2\,3)
\preceq \cdots \preceq (n-1\,n)$. We use the same notation for the
lexicographic order on $T^d$ for any $d \in {\mathbb N}$. Given two
saturated chains of the same length
${\mathcal C}_1, {\mathcal C}_2 \in \symm_n^d$, we write
${\mathcal C}_1 \preceq {\mathcal C}_2$ and say that ${\mathcal C}_1$
is {\em lexicographically smaller} than ${\mathcal C}_2$ if
$\lambda({\mathcal C}_1) \preceq \lambda({\mathcal C}_2)$. It is well
known, and easy to see (see, e.g., \cite[Chapter 5, Exercise 20]{BB}),
that $\preceq$ is a reflection order. (We refer the reader to, e.g.,
\cite[\S 5.2]{BB}, for the definition of and further information about
reflection orderings.) A saturated chain
$(x_0, \ldots , x_d) \in \symm_n^d$ is {\em increasing} if
$\lambda(x_0,x_1) \preceq \cdots \preceq \lambda(x_{d-1},x_d)$.

Let ${\mathcal C}:=(x_0, \ldots , x_d)$ be a saturated chain.  Let  $i
\in [d-1]$,  
and let $y_i \in \symm_n$ be the unique element such that $x_{i-1}
\lhd  y_i \lhd x_{i+1}$ and $y_i \neq x_i$.
Following \cite[\S 6]{BilBre} we define the {\em flip} of ${\mathcal
  C}$  at $i$ to be
\[
\flip_i({\mathcal C}):= (x_0, \ldots , x_{i-1},y_i,x_{i+1}, \ldots , x_d).
\]
Note that $\flip_i(\flip_i({\mathcal C}))={\mathcal C}$. 
The following result is essentially known. However, for lack of an
adequate  reference, and for
completeness, we include its proof here.
\begin{proposition}
	\label{pro:flipconn}
Let $u,v \in \symm_n$, $u \leq v$. Then any two maximal chains
  in $[u,v]$  are related by a sequence of flips.
\end{proposition}
\begin{proof}
It is well known (see, e.g., \cite[Proposition 4.3]{Dyer}) that
 there is a  unique increasing maximal chain
${\mathcal Z}$ in $[u,v]$, and that it is lexicographically first
among all  maximal chains in $[u,v]$. Let ${\mathcal C}=(x_0, \ldots , x_d)$ be a maximal
chain in $[u,v]$. It is enough to show that  ${\mathcal C}$ and ${\mathcal Z}$
are connected by a sequence of flips. We prove this by induction on
the  number of maximal chains that are lexicographically smaller than
${\mathcal C}$.  If ${\mathcal C} \neq  {\mathcal Z}$ then there is $i
\in [d-1]$ such that $\lambda(x_{i-1},x_i) \succ
\lambda(x_{i},x_{i+1})$. Let $(x_0, \ldots , x_{i-1},y_i,x_{i+1},
 \ldots , x_d) :=  \flip_i({\mathcal C})$.  
Then, since in $[x_{i-1},x_{i+1}]$ there is a unique increasing maximal chain, $\lambda(x_{i-1},y_i) \prec \lambda(y_{i},x_{i+1})$, and since this increasing maximal 
chain is lexicographically first among all  maximal chains in $[x_{i-1},x_{i+1}]$, 
$\lambda(x_{i-1},y_i) \prec \lambda(x_{i-1},x_{i})$. 
Hence $\flip_i({\mathcal C})$ is lexicographically smaller than ${\mathcal C}$, and this concludes the proof.
\end{proof}

\begin{proof}[Proof of Theorem~\ref{thm:bruhat}] 
	By Theorems~\ref{thm:bottomchain} and \ref{thm:topchain}, there exist $u,v \in \Perm_n(D)$ 
	such that $\Perm_n(D) \subseteq [u,v]$. By Proposition \ref{pro:maxchain} there is a 
	maximal chain ${\mathcal C}$ in $[u,v]$ such that ${\mathcal C} \subseteq \Perm_n(D)$.
	By Proposition \ref{pro:square}
 the flip of any maximal chain in $[u,v]$ that is contained in $\Perm_n(D)$ is still
 contained in $\Perm_n(D)$. Hence, by Proposition \ref{pro:flipconn}, all maximal
 chains in $[u,v]$ are contained in $\Perm_n(D)$, so $[u,v] \subseteq \Perm_n(D)$.
\end{proof}

Figure~\ref{fig:S4} shows the partition of $\symm_4$ into Bruhat intervals arising as odd diagram classes.

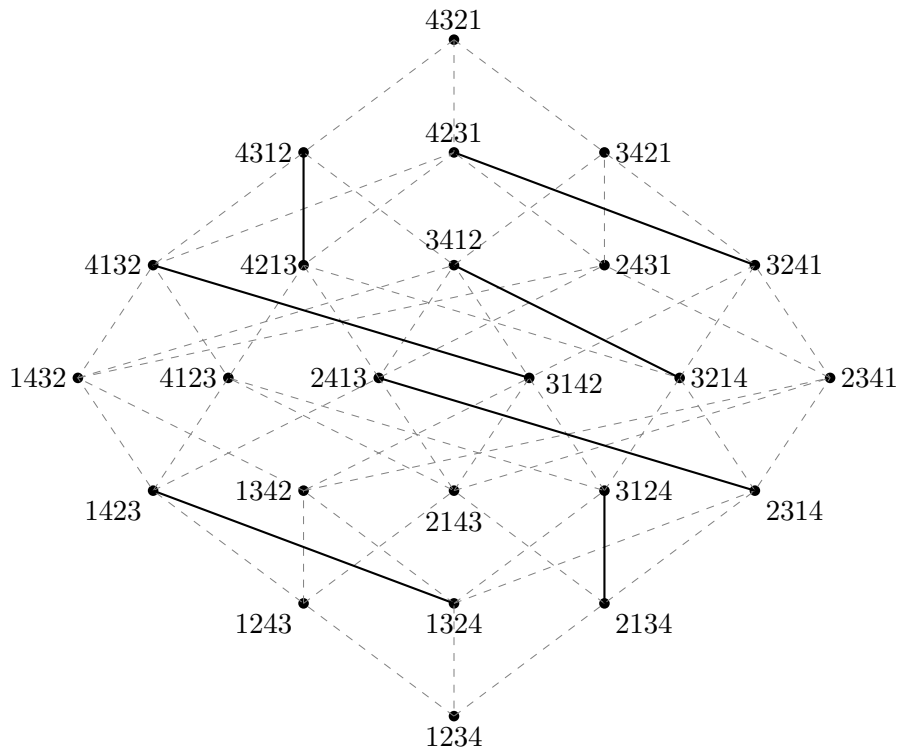
\begin{figure}[t]
\begin{tikzpicture}[scale=1]
\foreach \y in {0,9} {\fill[black] (0,\y) circle (2pt);}
\foreach \x in {-2,0,2} {\foreach \y in {1.5,7.5} {\fill[black] (\x,\y) circle (2pt);};}
\foreach \x in {-4,-2,0,2,4} {\foreach \y in {3,6} {\fill[black] (\x,\y) circle (2pt);};}
\foreach \x in {-5,-3,-1,1,3,5} {\fill[black] (\x,4.5) circle (2pt);}
\draw (0,0) coordinate (1234); \draw (0,9) coordinate (4321);
\draw (-2,1.5) coordinate (1243); \draw (0,1.5) coordinate (1324);
\draw (2,1.5) coordinate (2134); \draw (-2,7.5) coordinate (4312);
\draw (0,7.5) coordinate (4231); \draw (2,7.5) coordinate (3421);
\draw (-4,3) coordinate (1423); \draw (-2,3) coordinate (1342);
\draw (0,3) coordinate (2143); \draw (2,3) coordinate (3124);
\draw (4,3) coordinate (2314); \draw (-4,6) coordinate (4132);
\draw (-2,6) coordinate (4213); \draw (0,6) coordinate (3412);
\draw (2,6) coordinate (2431); \draw (4,6) coordinate (3241);
\draw (-5,4.5) coordinate (1432); \draw (-3,4.5) coordinate (4123);
\draw (-1,4.5) coordinate (2413); \draw (1,4.5) coordinate (3142);
\draw (3,4.5) coordinate (3214); \draw (5,4.5) coordinate (2341);
\draw[gray,dashed, very thin] (1234) -- (1243) -- (1423) -- (1432) -- (4132) -- (4312) -- (4321);
\draw[gray,dashed, very thin] (4321) --(3421) -- (3241) -- (2341) -- (2314) -- (2134) -- (1234) -- (1324);
\draw[gray,dashed, very thin] (1423) -- (4123) -- (4132) -- (4231) -- (4321);
\draw[gray,dashed, very thin] (1243) -- (1342) -- (1432) -- (3412) -- (4312);
\draw[gray,dashed, very thin] (1243) -- (2143) -- (4123) -- (4213);
\draw[gray,dashed, very thin] (1324) -- (1342) -- (3142);
\draw[gray,dashed, very thin] (1324) -- (3124) -- (4123);
\draw[gray,dashed, very thin] (1324) -- (2314);
\draw[gray,dashed, very thin] (2413) -- (4213) -- (4231);
\draw[gray,dashed, very thin] (2134) -- (2143) -- (2413) -- (3412) -- (3421);
\draw[gray,dashed, very thin] (3124) -- (3142) -- (3412);
\draw[gray,dashed, very thin] (1423) -- (2413) -- (2431) -- (4231);
\draw[gray,dashed, very thin] (1342) -- (2341) -- (2431) -- (3421);
\draw[gray,dashed, very thin] (2143) -- (3142) -- (3241);
\draw[gray,dashed, very thin] (2143) -- (2341);
\draw[gray,dashed, very thin] (3124) -- (3214) -- (4213);
\draw[gray,dashed, very thin] (2314) -- (3214);
\draw[gray,dashed, very thin] (1432) -- (2431);
\draw[gray,dashed, very thin] (3214) -- (3241);
\draw[thick] (3241) -- (4231);
\draw[thick] (3214) -- (3412);
\draw[thick] (4213) -- (4312);
\draw[thick] (3142) -- (4132);
\draw[thick] (2314) -- (2413);
\draw[thick] (1324) -- (1423);
\draw[thick] (2134) -- (3124);
\draw (1234) node[below] {\small $1234$};
\draw (4321) node[above] {\small $4321$};
\draw (2134) node[below right] {\small $2134$}; \draw (1324) node[below] {\small $1324$}; \draw (1243) node[below left] {\small $1243$};
\draw (1342) node[left] {\small $1342$}; \draw (1423) node[below left] {\small $1423$}; \draw (2143) node[below, yshift=-4pt] {\small$2143$}; \draw (2314) node[below right] {\small $2314$}; \draw (3124) node[ right] {\small $3124$};
\draw (1432) node[left] {\small$1432$}; \draw (3214) node[right] {\small$3214$}; \draw (4123) node[left] {\small$4123$}; \draw (2341) node[right] {\small$2341$}; \draw (2413) node[ left] {\small$2413$}; \draw (3142) node[right, xshift=2pt, yshift=-2pt] {\small$3142$};
\draw (4132) node[left] {\small$4132$}; \draw (4213) node[left, xshift=2pt] {\small$4213$}; \draw (3241) node[right] {\small$3241$}; \draw (2431) node[right] {\small$2431$}; \draw (3412) node[above, yshift=2pt] {\small$3412$};
\draw (4231) node[above] {\small$4231$}; \draw (4312) node[left] {\small$4312$}; \draw (3421) node[right] {\small$3421$};
\end{tikzpicture}
\caption{The partition of the symmetric group $\symm_4$ into odd diagram classes. Solid edges connect permutations within an odd diagram class. Each class in $\symm_4$ is either a singleton or a rank 1 Bruhat interval.}\label{fig:S4}
\end{figure}

\medskip

We conclude with a curious consequence of Theorem \ref{thm:bruhat}
that is, in a sense, dual to it, and with a conjecture.  It is clear
from the parity condition in Theorem~\ref{thm:legal swaps} that no
legal cover in Bruhat order is a covering relation in right weak
order, as these are given by adjacent transpositions. Theorem~\ref{thm:bruhat} has the following stronger
consequence.
\begin{corollary}\label{cor:antichain}
  Every odd diagram class $\Perm_n(D)$ is an antichain in right weak order.
\end{corollary}

\begin{proof}
   We denote by
$\leq _R $ the right weak order and by $\lcov _R $ the corresponding
covering relation. Suppose $u\sim v$ and $u\leq_R v$. Then $u$ cannot be covered by $v$, so there exists $w$ such that $u\lcov_R w \leq_R v$.  Since the same chain of relations holds in Bruhat order as well, by Theorem~\ref{thm:bruhat} this implies $w\sim u$, which is impossible. 
  \end{proof}
Computations with SageMath \cite{sagemath} suggest that  few isomorphism types of Bruhat intervals arise as odd diagram classes.
Moreover, based on  evidence for $n\leq 10$, we formulate the following.
\begin{conjecture}
  $\Perm_n(D)$ is rank-symmetric for any odd diagram $D$.
\end{conjecture}
Bruhat intervals arising as odd diagram classes are not, however, self-dual in general.
For example, if $D= \{ (1,1),(1,2),(1,3),(1,5),(2,4),(3,1),(3,2),(3,3),(5,2),(5,3),(7,3)\}$ and $n=9$ then $\Perm_9(D)=[654172839,958172634]$ and one can check that this interval is not self-dual.

\begin{ackn}
  This research was supported by the Swedish Research Council under
  grant no.\ 2016-06596 while the authors visited the Institut
  Mittag-Leffler in Djursholm, Sweden during the program ``Algebraic
  and enumerative combinatorics'' in~2020. FB was partially supported
  by the MIUR Excellence Department Project CUP~E83C18000100006. AC
  would like to thank Tobias Rossmann for help with some
  computations. Research of BT was partially supported by Simons
  Foundation Collaboration Grant for Mathematicians 277603, a
  University Research Council Competitive Research Leave from DePaul
  University, a DePaul University Faculty Summer Research Grant, and
  NSF Grant DMS-2054436.  The authors would like to thank the
  anonymous referees for their careful reading of the manuscript and
  helpful comments and suggestions.
  \end{ackn}

\end{document}